\documentclass[leqno, 12pt]{article}
\usepackage{amsmath}
\usepackage{amssymb}
\usepackage{amsthm}
\usepackage{amscd}
\usepackage{amsxtra}
\usepackage{verbatim}
\usepackage{xcolor}
\usepackage{color}
\usepackage{enumerate}
\usepackage{mathrsfs}
\usepackage{stmaryrd}
\usepackage[all]{xy} 

\usepackage{array,arydshln}
\usepackage{bm}

\allowdisplaybreaks

\numberwithin{equation}{section}

\definecolor{dblue}{rgb}{0,0,0.45}
\definecolor{red}{rgb}{0.7,0,0}

 \RequirePackage{geometry}
\newtheorem{theorem}{Theorem}[section]
\newtheorem{lemma}[theorem]{Lemma}
\newtheorem*{lemma*}{Lemma}
\newtheorem{corollary}[theorem]{Corollary}
\newtheorem{proposition}[theorem]{Proposition}

\theoremstyle{definition}

\newtheorem{remark}[theorem]{Remark}
\newtheorem{definition}[theorem]{Definition}

\theoremstyle{remark}

\newcommand{\N}{{\mathbb N}}
\newcommand{\R}{{\mathbb R}}

\newcommand{\cD}{{\mathcal D}}

\newcommand{\cH}{{\mathcal H}}

\newcommand{\cK}{{\mathcal K}}

\newcommand{\cM}{{\mathcal M}}
\newcommand{\cN}{{\mathcal N}}
\newcommand{\cP}{{\mathcal P}}

\newcommand{\cW}{{\mathcal W}}

\newcommand{\vp}{\varphi}

\newcommand{\la}{\langle}
\newcommand{\ra}{\rangle}
\newcommand{\nn}{\nonumber}
\newcommand{\ve}{\varepsilon}

\newcommand{\pfbM}{O (M; \cD \oplus \cD^{\perp})}

\newcommand{\vertiii}[1]{{\left\vert\kern-0.25ex\left\vert\kern-0.25ex\left\vert #1 
    \right\vert\kern-0.25ex\right\vert\kern-0.25ex\right\vert}}

\date{}
\begin{document}

\title{
Sub-Riemannian spectral distance
}
\author{   
Yuzuru \textsc{Inahama} 
}
\maketitle

\begin{abstract}
We study eigenvalues and eigenfunctions of 
the ``div-grad type" sub-Laplacian  with respect to Popp's volume
on a compact equiregular sub-Riemannian manifold $M$.
Since Popp's volume is canonically determined by 
the sub-Riemannian structure of $M$, the spetra of the sub-Laplacian
carry geometric meanings. In this paper, we first  
embed $M$ into the Hilbert space of square-summable 
sequences using eigenfunctions and then define a spectral distance
between two compact equiregular sub-Riemannian manifolds.
Our result is a sub-Riemannian analogue of B\'erard-Besson-Gallot's
classical work in the Riemannian case.
\end{abstract}

\section{Introduction}
\footnote[0]{
{\bf 2020 Mathematics Subject Classification.}
Primary 53C17; Secondary 58C40, 35K08, 58J65, 60H07.
{\bf Key words and phrases.}~ sub-Riemannian geometry, sub-Laplacian,
spectral embedding, 
spectral distance, stochastic differential equation.
The author is supported by JSPS KAKENHI Grant No. 26K06846. 
}
The study of eigenvalues 
and eigenfunctions of the Laplace-Beltrami operator 
on a compact Riemannian manifold is one of the most beautiful topics in
 Riemannian geometry. These spectral problems 
have been studied extensively and intensively.
One of many significant works on this topic is B\'erard-Besson-Gallot's 
spectral embedding theorem.
Using this, they also defined the spectral distance between 
two compact Riemannian manifolds. (See \cite{bbg}. Recently, 
this result was generalized to the case of RCD spaces by \cite{ah, honda}.)
Roughly speaking, these results can be summarized in the following way.

Denote by $0 =\lambda_0 < \lambda_1 \le \lambda_2 \le \cdots$ 
all the eigenvalues of the Laplace-Beltrami operator $\Delta_\cM$ on a 
connected compact Riemannian
manifold $\cM$ (without boundary)
in non-decreasing order counting the multiplicities.
Let the sequence $\{\vp_i\}_{i=0}^\infty$ of real-valued functions
be an orthonormal basis (ONB) of $L^2 (\cM)$ such that
$\Delta_\cM \vp_i =\lambda_i \vp_i$ for all $i \ge 0$.
They showed that the map 
\[
\cM \ni x \,\,\mapsto  \,\,  Z_M \{ e^{-\lambda_i t/2} \vp_i (x)\}_{i=1}^\infty \,\,\in \ell^2
\]
is an embedding for every $t>0$. Here, $Z_\cM>0$ is a suitable canstant and
$\ell^2$ is the Hilbert space of square-summable
sequences. (Moreover, this embedding 
approaches an isometry as $t\searrow 0$ if $Z_\cM$ is
suitably chosen.)
By taking the Hausdorff distance of images of the embeddings
and then varying all ONB's of eigenfunctions as above, 
they defined a distance $\mathbf{dist}_t (\cM, \hat{\cM})$ for each $t>0$.
Since they proved that $\mathbf{dist}_t (\cM, \hat{\cM})=0$ if and only if 
$\cM$ and $\hat{\cM}$ are isometric as Riemannian manifolds,
$\mathbf{dist}_t$ is actually a distance and called the spectral distance.

In this paper, we generalize the above-mentioned results to the case of 
equiregular sub-Riemannian manifolds.
An equiregular sub-Riemannian manifold has a natural measure which is determined by its geometric structure only.
Hence, the spectra of 
the ``div-grad type" sub-Laplacian are supposed to have valuable geometric information.
This is why spectral geometry for equiregular sub-Riemannian manifolds
looks quite intriguing.
Of course, Riemannian manifolds are very special examples of
equiregular sub-Riemannian manifolds.

Our main results in this paper are twofold.
Firstly, we prove that B\'erard-Besson-Gallot's embedding theorem still holds
for compact equiregular sub-Riemannian manifolds.
The most difficult part of the proof (to the author) 
is an $L^\infty$-type estimate of the first-order derivatives of 
eigenfunctions, which is shown by probabilistic methods.
Secondly, we generalize the spectral distance to the case of 
 compact connected equiregular sub-Riemannian manifolds
  and then prove that the distance of 
 two such manifolds equals zero if and only if they are isometrically 
 isomorphic as sub-Riemannian manifolds.
Our proofs basically follow those in \cite{bbg}. It should be noted that
we do not use Varadhan's asymptotics.

The organization of this paper is as follows. 
In Section 2, we recall the basics of sub-Riemannian geometry, which
include the equiregularity, Popp's volume, the div-grad type
sub-Laplacian and the cometric.
In Section 3, we prove several technical lemmas for later use.
In Section 4, we introduce the spectral embedding and 
the spectral distance in an analogous way to the classical Riemannian case.
In Section 5, the non-degeneracy of the spectral distance is shown.
Appendix is devoted to proving a key technical proposition 
(Proposition \ref{prop.ef_est2}). Our proof is probabilistic.

\section{Setting}

In this section we recall the basics of sub-Riemannian geometry,
following nice textbooks such as \cite{abb, cc, mo, ri}.
We say that 
 $(M,\cD,g)$ is a sub-Riemannian manifold if   
 \begin{enumerate}
\item[(i)] $M$ is a connected, smooth manifold
(without boundary) of dimension $d$, 
\item[(ii)] $\cD\subset TM$, $TM$ being the tangent bundle of $M$,
is a smooth distribution of constant rank $n~(1 \le n \le d)$
which satisfies the H\"ormander condition at every $x \in M$,
\item[(iii)] $g=(g_x)_{x\in M}$,
where each $g_x$ is an inner product on the fiber $\cD_x$
and $x\mapsto g_x$ is smooth as a function of $x$.
\end{enumerate}
When there is no risk of confusion, we simply say that 
$M$ is a sub-Riemannian manifold.
Throughout this paper, $M$ is assumed to be compact.

The precise statement of 
 the H\"ormander condition on $\cD$ at $x \in M$ is as follows:
Define $\cD_0(x)=\{0\}$,
$\cD_1 (x) =\cD (x)$ and
\[
\cD_{k} (x)
=
\mbox{linear span of } 
\Bigl\{  
 \underbrace{
 [A_1,    [\ldots,  [A_{l-1}, A_l]] ] 
 }_{(l-1) {\rm brackets}}
 (x)
 \, \Big\vert\,
 1 \le l \le k,  \,  A_{1}, \ldots, A_l \in \Gamma (\cD)
 \Bigr\}
\]
for $k \ge 2$.
Here, $\Gamma (\cD)$ stands for the $C^\infty$-module
 of smooth sections of $\cD$ over $M$.
We say that 
$\cD$ satisfies the H\"ormander condition at $x$ 
if there exists $N =N(x)$ such that 
$\cD_N (x)= T_x M$.

A sub-Riemannian manifold
 $(M,\cD,g)$ is said to be equiregular 
 if $\dim \cD_k(x)$ is constant in $x \in M$ for all $k \ge 1$.
The smallest constant $N_0$ such that 
$\cD_{N_0} (x)= T_x M$ is called the {\it step} of the H\"ormander condition.
In this case, 
\[
Q :=\sum_{k=1}^{N_0} k(\dim \cD_k(x)-\dim\cD_{k-1}(x))
\]
is also constant in $x$ and equals the Hausdorff dimension 
of $M$ with respect to the usual sub-Riemannian distance on $M$.
Denote Popp's volume on $M$ by $\mu$.
This is determined by the equiregular sub-Riemannian structure only
and is a smooth measure on $M$ in the sense that
its restriction to every local coordinate chart 
is written as a strictly positive smooth density function 
times the Lebesgue measure on the chart.
(For Popp's volume, see \cite[Chapter 20]{abb} or \cite[Section 10.6]{mo}.)

We study the second-order differential operator of the form
$\Delta =(\nabla^{\cD})^* \nabla^{\cD}$,
where $\nabla^{\cD}$ is the horizontal gradient in the direction of $\cD$
and $(\nabla^{\cD})^*$ is the adjoint of $\nabla^{\cD}$ with respect to $\mu$.
(In our sign convention, the sub-Laplacian 
$\Delta$ is a non-negative operator on $L^2 (M)$.)
By the way it is defined, $\Delta$ with its domain $C_0^{\infty} (M)$
is clearly symmetric on $L^2 (M)$.
Since $M$ is compact,  $\Delta$ is known to be 
essentially self-adjoint on $C^{\infty} (M)$ (whose
unique self-adjoint extension will be denoted by the same symbol again)
and
$e^{-t \Delta}$ is of trace class for every $t >0$, 
where $(e^{-t \Delta})_{t \ge 0}$ is the heat semigroup
 associated with $ \Delta$.

Needless to say, 
a connected  Riemannian manifold is a special example of 
equiregular sub-Riemannian manifold. In that case, Popp's volume $\mu$
coincides with the usual Riemannian measure and the associated
sub-Laplacian $\Delta$ coincides with the Laplace-Beltrami operator.


Denote by $g^*$ the cometric of $(M,\cD,g)$.
At $x\in M$, the cometric $g^*_x \in (T_x^* M)^* \otimes (T_x^* M)^*
=T_xM \otimes T_xM$ is defined  by
\[
g^*_x \la \xi, \eta\ra = \sum_{i=1}^n \la \xi, v_i\ra \la \eta, v_i\ra,
\qquad \xi,\eta \in T^*_x M, 
\]
where $\{v_i\}_{i=1}^n$ is an (or any) ONB of $\cD_x\subset T_xM$.
Hence, $g^* \in \Gamma (T_xM \otimes T_xM)$.

It is well-known that 
one can recover $(\cD_x, g_x)$ from $g^*_x$ as follows.
First, $\cD_x = (\ker g^*_x)^\perp$ holds, where
$\ker g^*_x := \{ \xi \in T_x^*M \mid g^*_x \la \xi, \xi\ra=0\}$.
Moreover, one can easily see that
\[
g_x(v,v) = \sup
  \left\{ 2 \la  \xi, v\ra -g^*_x \la \xi, \xi\ra \mid 
   \xi \in T^*_x M \right\}, \qquad v \in \cD_x.
\]
Therefore, identifying the sub-Riemannian structure
is equivalent to identifying the cometric.

One can obtain the cometric from the sub-Laplacian $\Delta$.
Let $\{V_i\}_{i=1}^n$ is a local orthonormal frame 
of $\cD$ around $x_0\in M$. Then, the sub-Laplacian $\Delta$ writes
\[
-\Delta =\sum_{i=1}^n V_i^2 + \mbox{($1$st order differential operator)}.
\]
Hence, for any $\psi, \chi \in C^\infty (M)$ vanishing at $x_0$, we have 
\[
-\Delta (\psi \chi) \vert_{x=x_0} 
= 
2
\sum_{i=1}^n \la d\psi (x_0), V_i (x_0)\ra \la d\chi (x_0),  V_i (x_0)\ra
=
2 g^*_{x_0} \la d\psi (x_0), d\chi (x_0)\ra,
\]
where $d\psi$ and $d\chi$ are the exterior derivative of $\psi$ and $\chi$, respectively.
Since any covector $\xi$ at $x_0$ can be written as $d\psi(x_0)$ for some $\psi$,
we can obtain $g^*$ from $\Delta$.


Now let us introduce the heat kernel $p=p (t,x,y)$,
which is a smooth function
on $(0,\infty)\times M \times M$ and satisfies 
 \[
 e^{-t \Delta}f (x) = \int_M p (t,x,y) f(y) \mu (dy), \qquad f\in L^2 (M).
 \]
Since $\Delta$ is self-adjoint, $p (t,x,y) = p(t,y,x)$.
By Mercer's theorem, we have ${\rm Trace} (e^{-t \Delta}) = \int_M p (t,x,x) \mu (dx)$ for all $t>0$.
The following Weyl-type asymptotics for the heat trace is known: 
\begin{equation}\label{asy.trc}
{\rm Trace} (e^{-t \Delta})  \sim 
\frac{c_0}{t^{Q/2}} 
\qquad
\mbox{as $t \searrow 0$}
\end{equation}
for some constant $c_0 >0$.  (See \cite{me, inatan2} for example.)
Here, the asymptotic symbol
``$\sim$" means that the quotient converges to $1$.

Denote by $0 =\lambda_0 < \lambda_1 \le \lambda_2 \le \cdots \nearrow \infty$ be all the eigenvalues of $\Delta$ in non-decreasing order
counting the multiplicities.
Due to Chow-Reshevskii's theorem, the lowest eigenvalue $\lambda_0=0$
is necessarily simple and the corresponding eigenspace consists of 
constant functions only.
Set the eigenvalue counting function 
$\mathcal{N}_{\Delta} (\lambda) :=\# \{i \ge 0 \mid \lambda_i \le \lambda\}$ for $\lambda \ge 0$.
By Karamata's Tauberian theorem, we have from \eqref{asy.trc} that 
\begin{equation}\label{asy.numb_e.v.}
 \mathcal{N}_{\Delta} (\lambda)  \sim 
\frac{c_0}{\Gamma (Q/2 +1)} \lambda^{Q/2}
\qquad
\mbox{as $\lambda \to\infty$}.
\end{equation}
Here, $\Gamma$ stands for the usual Gamma function.

Let the sequence $\{\vp_i\}_{i=0}^\infty$ of real-valued functions
be an ONB of $L^2 (M)$ such that
$\Delta\vp_i =\lambda_i \vp_i$ for all $i \ge 0$
and $\vp_0 \equiv \mu (M)^{-1/2}$.
Thanks to the hypoellipticity, $\vp_i$'s are necessarily smooth. 
The totality of such ONB's is denoted by $\mathcal{B} (M, \Delta)$.
Then, we have
\begin{equation}\label{hk_L^2.eq}
p(t,x,y) = \sum_{i=0}^\infty e^{-\lambda_i t} \vp_i (x)\vp_i (y).
\end{equation}
At the moment, 
the convergence on the right hand side (RHS) of \eqref{hk_L^2.eq} takes place 
in $L^2 (M\times M)$ for every fixed $t>0$.
So, it is not so clear a priori whether RHS makes sense
for a fixed $(x,y)$.
In the next section, we will check that this sum actually converges 
uniformly and therefore \eqref{hk_L^2.eq} makes pointwise sense.

\section{Preliminary lemmas}

In this section we provide several lemmas on eigenvalues, 
eigenfunctions and the heat kernel associated with the sub-Laplancian.

\begin{lemma}\label{lem.int_EVCfunct}
For every $r>0$ and $t >0$, we have
\[
\sum_{i=0}^\infty \lambda_i^r e^{-\lambda_i t} <\infty.
\]
Similarly, if $q>Q/2 +1$, we then have
\[
\sum_{i=1}^\infty \lambda_i^{-q}  <\infty.
\]
\end{lemma}

\begin{proof}
We prove the first assertion.
One can easily see from \eqref{asy.numb_e.v.} that
\begin{align*}
\sum_{i=0}^\infty \lambda_i^r e^{-\lambda_i t} 
&=
\sum_{\lambda_i \le 1} \lambda_i^r e^{-\lambda_i t} 
+
\sum_{n=1}^\infty \sum_{n<\lambda_i \le n+1} \lambda_i^r e^{-\lambda_i t} 
\\
&\le 
\mathcal{N}_{\Delta} (1) 
+ \sum_{n=1}^\infty (n+1)^re^{-n t}\mathcal{N}_{\Delta} (n+1)
\\
&\le 
\mathcal{N}_{\Delta} (1) 
+C \sum_{n=1}^\infty (n+1)^{r +Q/2}e^{-n t}<\infty
\end{align*}
for some constant $C>0$ independent of $n$.

We prove the second assertion. By similar computation as above, we have
\begin{align*}
\sum_{i=1}^\infty \lambda_i^{-q}  
&=
\sum_{\lambda_i \le 1} \lambda_i^{-q}
+
\sum_{n=1}^\infty \sum_{n<\lambda_i \le n+1} \lambda_i^{-q}
\\
&\le 
\lambda_1^{-q}\mathcal{N}_{\Delta} (1) 
+
C \sum_{n=1}^\infty n^{-q} (n+1)^{Q/2}<\infty.
\end{align*}
Here we used that $q -Q/2 >1$.
\end{proof}


\begin{proposition}\label{prop.ef_est}
Let $i \ge 1$ and $\psi_i$ be an (real-valued) eigenfunction of $\Delta$ 
associated with the eigenvalue $\lambda_i$ with $\| \psi_i\|_{L^2}=1$. 
Then, there exists a constant $C>0$ such that 
\[
\|\psi_i\|_\infty \le C \lambda_i^{Q/4},  \qquad \mbox{$i \ge 1$.}
\]
Here,  $\|\,\cdot\,\|_\infty$ is the usual sup-norm
and $C$ does not depend on $i$ or $\psi_i$. 
\end{proposition}

\begin{proof}
Since $\psi_i$ is an eigenfunction of $e^{-t \Delta}$
with eigenvalue $e^{-\lambda_i t}$,  we have
\begin{align*}
 |e^{-\lambda_i t}\psi_i (x)|
 &=
 \left| \int_M p (t,x,y) \psi_i (y) \mu (dy) \right|
 \\
 &\le 
 \| \psi_i\|_{L^2}
 \left\{ \int_M p (t,x,y)^2 \mu (dy)\right\}^{1/2}
 \le 
 p (2t,x,x)^{1/2}.
\end{align*}
Here, we used the symmetry of $p$ and the Chapman-Kolmogorov formula.

Uniform on-diagonal short-time asymptotics of the heat kernel 
was proved in \cite{me, inatan2}.  
It claims that there exists a constant $c_1 >0$,
which is independent of $t$ and $x$, such that
\[
p (t,x,x) \le \frac{c_1}{t^{Q/2}}, \qquad  (t, x) \in (0,1] \times M.
\]
By putting $t =1/\lambda_i$, we have
\[
|\psi_i (x)| \le e \sqrt{c_1} (2\lambda_i)^{Q/4}  \qquad 
\mbox{if $\lambda_i \ge 1$.}
\]
Since there are only finitely many $\lambda_i$'s less than $1$,
the proof is finished.
\end{proof}


\begin{lemma}\label{lem.HK_unif}
The series on RHS of 
\eqref{hk_L^2.eq}  converges absolutely and 
uniformly on $[s, \infty) \times M\times M$ for every $s>0$.
\end{lemma}

\begin{proof}
From Proposition \ref{prop.ef_est}, we can see that 
\[
\sum_{i=1}^\infty \left| e^{-\lambda_i t} \vp_i (x)\vp_i (y)\right|
\le 
C^2\sum_{i=1}^\infty e^{-\lambda_i s} \lambda_i^{Q/2},
\qquad
(t,x,y) \in [s, \infty) \times M\times M.
\] 
Here, $C>0$ is the constant in Proposition \ref{prop.ef_est}.
By Lemma \ref{lem.int_EVCfunct}, RHS 
is finite and independent of $(t,x,y)$, from which our assertion follows.
\end{proof}


Since $\{\vp_i\}_{i=0}^\infty \in \mathcal{B} (M, \Delta)$ is an ONB of $L^2 (M)$, every $f \in L^2 (M)$ admits the following expansion in 
the $L^2$-topology:
\begin{equation} \label{eq.Fourier}
f (x) = \sum_{i=0}^\infty \langle f, \vp_i\rangle \vp_i (x)
\end{equation}
where $\langle f, \vp_i\rangle =\int_M f \vp_i d\mu$ stands for 
the inner product of $L^2 (M)$.
When $f$ is regular enough, the series in \eqref{eq.Fourier}
converges absolutely and uniformly
and \eqref{eq.Fourier} makes sense at every point $x$.

\begin{lemma} \label{lem.Fourier}
Suppose that $f \in C^\infty (M)$ and $k \in \N :=\{1,2, \ldots\}$.
 Then,  we have 
 \[
 |\langle f, \vp_i\rangle| \le \lambda_i^{-k} \| \Delta^k f\|_{L^2},
 \qquad i \ge 1. 
 \]
In particular, the series in \eqref{eq.Fourier}
converges absolutely and uniformly on $M$.
\end{lemma}

\begin{proof} 
It is easy to see from the self-adjointness of $\Delta$ that
\[
 \lambda_i^{k}\langle f, \vp_i\rangle 
 = \langle f,  \Delta^k \vp_i\rangle
 =\langle  \Delta^k  f,  \vp_i\rangle.
\]
Then, 
\[
 |\lambda_i^{k}\langle f, \vp_i\rangle|^2 
 \le 
 \sum_{j=0}^\infty \langle  \Delta^k  f,  \vp_j\rangle^2 
 = \| \Delta^k f\|_{L^2}^2.
\]
Thus, we obtained the first assertion.

We prove the second assertion. 
It follows from the first assertion
and Proposition \ref{prop.ef_est} that
\[
 \sum_{i=1}^\infty |\langle f, \vp_i\rangle| \, \|\vp_i \|_{\infty}
 \le 
 C\| \Delta^k f\|_{L^2} \sum_{i=1}^\infty \lambda_i^{ -k +Q/4}<\infty
\]
if $k > Q+1$.
Here, we used Lemma \ref{lem.int_EVCfunct}.
\end{proof}


\begin{corollary} \label{cor.LapFour}
For every $f \in C^\infty (M)$, it holds that
\begin{equation} \label{eq.Fourier2}
\Delta f (x) = \sum_{i=0}^\infty \lambda_i \langle f, \vp_i\rangle \vp_i (x)
\end{equation}
and that the series in \eqref{eq.Fourier2}
converges absolutely and uniformly on $M$.
\end{corollary}

\begin{proof} 
The absolute and uniform convergence can be shown 
in the same way as in Lemma \ref{lem.Fourier}.
Hence, it also converges in $L^2 (M)$. 
Obviously, we have 
\[
\Delta  \left[
\sum_{i=0}^m  \langle f, \vp_i\rangle \vp_i 
\right]
= \sum_{i=0}^m \lambda_i \langle f, \vp_i\rangle \vp_i 
\]
for all $m\in \N$. Since $\Delta$ is a closed operator on $L^2 (M)$,
we may let $m\to\infty$ to obtain 
\[
\Delta f=
\Delta  \left[
\sum_{i=0}^\infty  \langle f, \vp_i\rangle \vp_i 
\right]
= \sum_{i=0}^\infty \lambda_i \langle f, \vp_i\rangle \vp_i
\]
in $L^2 (M)$. Since both sides are continuous in $x$, 
\eqref{eq.Fourier2} was proved.
\end{proof}


\begin{corollary} \label{cor.separate}
Let $\{\vp_i\}_{i=0}^\infty \in \mathcal{B} (M, \Delta)$.
Then, any two points $x, y \in M~(x\neq y)$ are separated by
$\{\vp_i\}_{i=0}^\infty$.
\end{corollary}

\begin{proof} 
By Lemma \ref{lem.Fourier}, the expansion \eqref{eq.Fourier}
makes pointwise sense.
Therefore, if $\vp_i (x) = \vp_i (y)$ for all $i \ge 0$, 
we have $f (x) =f(y)$ for all $f\in C^\infty (M)$. 
This is a contradiction.
\end{proof}


Now we estimate the first-order derivative of eigenfunctions of $\Delta$.
Technically, this lemma is quite important.
Unfortunately, the author has no precise information on the exponent $\nu$.

\begin{proposition}\label{prop.ef_est2}
Let $i \ge 1$ and $\psi_i$ be an (real-valued) eigenfunction of $\Delta$ 
associated with the eigenvalue $\lambda_i$ with $\| \psi_i\|_{L^2}=1$. 
Let $A$ be a smooth vector field on $M$.
Then, there exist constants $C>0$ and $\nu >0$ such that 
\[
\|A \psi_i\|_\infty \le C \lambda_i^{\nu},  \qquad \mbox{$i \ge 1$.}
\]
Here, {\rm (i)} $C$ does not depend on $i$ or $\psi_i$
and  {\rm (ii)} $\nu$ does not depend on $A$, $i$ or $\psi_i$.
\end{proposition}

\begin{proof} 
We use Proposition \ref{prop.Linfty_est}, which provides 
an $L^\infty$-estimate of the first-order derivative of the heat semigroup.

Since $\psi_i$ is an eigenfunction of $e^{-t \Delta}$
with eigenvalue $e^{-\lambda_i t}$,  we have $e^{-\lambda_i t}\psi_i 
= e^{-t \Delta}\psi_i $.
Apply $A$  to both sides and use Proposition \ref{prop.Linfty_est}.
Then, we have
\[
e^{-\lambda_i t}\|A\psi_i \|_\infty = \| A e^{-t \Delta}\psi_i \|_\infty
\le 
C_1 t^{-\nu_1}\|\psi_i \|_\infty
\le 
C_2 t^{-\nu_1}  \lambda_i^{Q/4}, \qquad t \in (0,1], \, i \ge 1
\]
where $C_1, C_2$ an $\nu_1$  are positive constants independent of 
$i$, $\psi_i$, and $t$.
For the last inequality, we used
Proposition \ref{prop.ef_est}.

If $\lambda_i \ge 1$, set $t = 1/\lambda_i$ in the above inequality.
Then,  $\|A\psi_i \|_\infty \le e C_2 \lambda_i^{\nu_1 +Q/4}$.
Since there are only finitely many $\lambda_i$'s belonging to $(0,1)$,
we are done.
\end{proof}


In the following corollary, 
$A^{(x)}p(t,x,y)$ stands for the application of $A\in \Gamma (TM)$ to the function
$x \mapsto p(t,x,y)$ for each fixed $(t, y)$.  
Also, $A^{(x)}\tilde{A}^{(y)}p(t,x,y)$ is defined in an analogous way.

\begin{corollary} \label{cor.Diff_unif_conv}
Let $A$ and $\tilde{A}$ be smooth vector fields on $M$.
Then, for every fixed $t>0$, we have 
\begin{align}  
A^{(x)} p(t,x,y) &= \sum_{i=1}^\infty e^{-\lambda_i t}( A\vp_i) (x)\vp_i (y),
\label{eq.nohoho1}
\\
A^{(x)}\tilde{A}^{(y)}p(t,x,y) &= \sum_{i=1}^\infty e^{-\lambda_i t} ( A\vp_i) (x) (\tilde{A}\vp_i) (y)
\label{eq.nohoho2}
\end{align}
for every $(t, x, y) \in (0,\infty)\times M \times M$.
Moreover, the series in \eqref{eq.nohoho1} and that in \eqref{eq.nohoho2}
both converge absolutely and uniformly on $[s ,\infty)\times M \times M$
for every $s>0$.
\end{corollary}

\begin{proof} 
We prove \eqref{eq.nohoho1}. Take any $x\in M$. 
Then, there exists a coordinate neighborhood $\{U, (x^1, \ldots, x^d)\}$ of $x$
such that $\partial/ \partial x^j$ extends to a global vector field  on $M$
for all $j~(1\le j \le d)$. 
(The extended vector fields are denoted by the same symbols again.)
By taking $U$ smaller if necessary, we may also assume 
that $A$ can be written as 
\[
A (x) = \sum_{j=1}^d a_j (x) \frac{\partial}{\partial x^j}, \qquad x \in U,
\]
where $a_j$'s are certain smooth and bounded functions on $U$ for all $j~(1\le j \le d)$.

By Lemma \ref{lem.int_EVCfunct},
Propositions \ref{prop.ef_est} and \ref{prop.ef_est2}, 
it holds for all $j$ that
\begin{align}  \nn 
\sum_{i=0}^\infty e^{-\lambda_i t}
\left|\frac{\partial \vp_i }{\partial x^j} (x)\right| \, |\vp_i (y)|
&\le 
c_1 \sum_{i=0}^\infty e^{-\lambda_i s} \lambda_i^{\nu + Q/4} <\infty,
\quad 
(t,x,y) \in [s, \infty)\times U \times M.
\end{align}
 Here, $c_1 >0$  is a constant independent of $i$.
So, this above absolute series  converges uniformly on $[s, \infty)\times U \times M$, which implies that 
 \[
 \left.  \frac{\partial}{\partial x^j} p(t, \,\cdot, y) \right|_{x}
 = 
 \sum_{i=0}^\infty e^{-\lambda_i t}
 \frac{\partial \vp_i }{\partial x^j} (x) \,\vp_i (y)
\qquad \mbox{on $[s, \infty)\times U\times M$.}
\]
Multiplying both side by $a_j (x)$ and then summing over $j$, 
we obtain \eqref{eq.nohoho1} on $[s, \infty)\times U\times M$.
Since $M$ is compact, $M$ can be covered by finitely many such $U$'s.
Thus, we have shown \eqref{eq.nohoho1}.

The proof of \eqref{eq.nohoho2} is essentially the same. So, we omit it.
\end{proof}

\section{Spectral embedding and spectral distance}

In this section we define spectral embeddings and spectral distances 
in an analogous way to B\'erard-Besson-Gallot \cite{bbg}. 

\begin{definition} \label{def.embed}
For $t>0$ and
$\mathbf{a} =\{\vp_i^{\mathbf{a}}\}_{i=0}^\infty\in \mathcal{B} (M, \Delta)$,
we define a map $I_t^{\mathbf{a}}\colon M \to \ell^2$ by 
\[
I_t^{\mathbf{a}} (x):= Z_M \{ e^{-\lambda_i t/2} \vp_i^{\mathbf{a}} (x)\}_{i=1}^\infty, 
\]
where we set $Z_M := \sqrt{\mu (M)} >0$.
\end{definition}

Note that $i=0$ is excluded from the above definition.
By Corollary \ref{cor.separate}, $I_t^{\mathbf{a}}$ is injective.
Moreover, the map 
\[
(0, \infty)\times M \ni (t, x) \mapsto I_t^{\mathbf{a}} (x)\in \ell^2
\]
is continuous.  Indeed, noting that $\vp_0^{\mathbf{a}}$ is constant, we
can easily see that
\begin{align*}  
Z_M^{-2} \| I_{t_n}^{\mathbf{a}} (x_n)- I_t^{\mathbf{a}} (x) \|_{\ell^2}^2
&=
\sum_{i=1}^\infty \{ e^{-\lambda_i t_n/2} \vp_i^{\mathbf{a}} (x_n)-e^{-\lambda_i t/2} \vp_i^{\mathbf{a}} (x)\}^2
\\
&=
p(t_n, x_n, x_n) +p(t, x, x) -2p((t_n+t)/2, x_n, x) \to 0 
\end{align*}
if $(t_n, x_n)\to (t,x)$ as $n\to \infty$.
Therefore, $I_t^{\mathbf{a}} (M)$ is compact 
and $I_t^{\mathbf{a}}$ is a homeomorphism from $M$ to $I_t^{\mathbf{a}} (M)$
for every $t>0$.


\begin{lemma} \label{lem.embed_ell^2}
For every $t>0$, the injection $I_t^{\mathbf{a}}\colon M \hookrightarrow \ell^2$
is an embedding.
\end{lemma}

\begin{proof} 
Suppose that $v \in T_z M$ satisfies that
\[
\langle dI_t^{\mathbf{a}}, v\rangle = Z_M \left\{e^{-\lambda_i t /2} 
\langle d\vp_i^{\mathbf{a}}, v\rangle \right\}_{i=1}^\infty =\mathbf{0},
\]
where $d$ stands for the exterior derivative on $M$.
Obviously, this is equivalent to $\langle d\vp_i^{\mathbf{a}}, v\rangle =0$ for all $i\ge 1$.
It is enough to show $v=0$ from this condition.
(Note that the left-most side is $\ell^2$-valued differentiation
and that the left equality can be verified by 
Proposition \ref{prop.ef_est2} and the dominated convergence theorem 
for the infinite sum.)

Take a smooth vector field $V$ on $M$ such that $V_z =v$.
Assume that $f \in C^\infty (M)$ and $k$ is large enough.
Then, we can see 
from Lemma \ref{lem.Fourier}, Proposition \ref{prop.ef_est2}
and Lemma \ref{lem.int_EVCfunct} that 
\[
\sum_{i=0}^\infty |\langle f, \vp_i^{\mathbf{a}}\rangle |  \|V\vp_i^{\mathbf{a}} \|_{\infty} 
\le 
C \sum_{i=1}^\infty    \lambda_i^{-k} \| \Delta^k f\|_{L^2} \cdot C \lambda_i^{\nu}
<\infty
\]
and that
\[
(Vf) (x) = \sum_{i=1}^\infty \langle f, \vp_i^{\mathbf{a}}\rangle (V\vp_i^{\mathbf{a}}) (x),
\qquad x\in M.
\]
Note that the above series converges absolutely and uniformly on $M$.
Hence, we may evaluate it at the given point $z$  to obtain
$\langle df, v\rangle = \sum_{i=0}^\infty \langle f, \vp_i^{\mathbf{a}}\rangle
\langle d\vp_i^{\mathbf{a}}, v\rangle =0$.
Since $f$ is arbitrary, we obtain $v=0$.
\end{proof}


Now we introduce a pseudometric between two 
equiregular sub-Riemannian manifolds.
As we will see below, this becomes a metric (distance).
Therefore, we will call it the sub-Riemannian spectral distance.

\begin{definition} \label{def.dist_SR}
For $t>0$ and 
two compact equiregular sub-Riemannian manifolds
$(M,\cD,g)$ and $(\hat{M}, \hat{\cD}, \hat{g})$, we set
\begin{align*}
\mathbf{dist}_t (M, \hat{M}) &:=\max \left\{  
\sup_{\mathbf{a}\in \mathcal{B} (M, \Delta)} 
\inf_{\mathbf{b}\in \mathcal{B} (\hat{M}, \hat{\Delta})}
\mathrm{HD} (I_t^{\mathbf{a}} (M), I_t^{\mathbf{b}} (\hat{M})) \right., 
\\
&\qquad\qquad\qquad\qquad
\left.
\sup_{\mathbf{b}\in \mathcal{B} (\hat{M}, \hat{\Delta})}
\inf_{\mathbf{a}\in \mathcal{B} (M, \Delta)} 
\mathrm{HD} (I_t^{\mathbf{a}} (M), I_t^{\mathbf{b}} (\hat{M}))
\right\}.
\end{align*}
Here, $\mathrm{HD}$ stands for the Hausdorff distance 
for compact subsets of $\ell^2$.
\end{definition}


It is trivial that $\mathbf{dist}_t$ satisfies the triangle inequality.
Obviously, $\mathbf{dist}_t (M, \hat{M}) =0$ 
if $(M,\cD,g)$ and $(\hat{M}, \hat{\cD}, \hat{g})$
are isomorphic as sub-Riemannian manifolds.
In what follows, we will show that the converse is also true.


Consider a compact equiregular sub-Riemannian manifold 
$(M,\cD,g)$ and its sub-Laplacian $\Delta$ again.
Denote by $0 =\lambda_0^\prime < \lambda_1^\prime <\lambda_2^\prime < \cdots$
be all the eigenvalues of $\Delta$ in {\it strictly increasing} order
{\it without} counting the multiplicities.
We denote by $E_j$ 
the eigenspace corresponding to $\lambda_j^\prime$.
Clearly, $\dim E_j <\infty$ for all $j \ge 0$ and
$L^2 (M, \mu) = \oplus_{j=0}^\infty E_j$.
It immediately follows  that
$\mathcal{B} (M, \Delta)= \{\vp_0\}\times \prod_{j=1}^\infty \mathcal{B} (E_j)$.
Here, $\mathcal{B} (E_j)$ denotes the set of all 
ONB's of $E_j$ and therefore can naturally be identified with 
the orthogonal group $O (\dim E_j)$.
By Tychonoff's theorem, $\mathcal{B} (M, \Delta)$ is compact.

On $O(k)$, $k \ge 1$, the distance $\tilde{\rho}_k$ defined by 
the Frobenius norm
$\tilde{\rho}_k (A,B) :=\|A-B\|_{{\rm Fr}}$ induces the usual topology 
and satisfies $\tilde{\rho}_k (A,B) \le 2\sqrt{k}$ for all $A, B\in O(k)$.
Moreover, $\tilde{\rho}_k$ is bi-invariant.

This distance viewed as one on $\mathcal{B} (E_j)$ through 
the above identification is denoted by $\rho_{E_j}$, where 
$k=\dim E_j$ is assumed.
We introduce a distance $\rho$ on $\mathcal{B} (M, \Delta)$: 
\[
\rho (\mathbf{a}, \mathbf{b})^2 :=\sum_{j=1}^\infty (\lambda_j^\prime)^{-N}
\rho_{E_j} ( \mathbf{a}\vert_{E_j}, \mathbf{b}\vert_{E_j})^2,
\qquad \mathbf{a}, \mathbf{b} \in \mathcal{B} (M, \Delta).
\]
Now, let us check that the series on RHS converges
if $N =N_M$ is large enough. 
We can easily see that 
\begin{align*}
\sum_{j=1}^\infty (\lambda_j^\prime)^{-N} 4\dim E_j
\le
4\sum_{i=1}^\infty \lambda_i^{-N},
\end{align*}
which is finite if $N > Q/2 +1$ (see Lemma \ref{lem.int_EVCfunct}).
In what follows, we set $N = Q/2 +2$.
Then, one can easily see that $\rho$ is a distance that 
induces the product topology of $\mathcal{B} (M, \Delta)$.

\begin{proposition} \label{prop.diff_embed}
Let the notation be as above. Then, for all $t, s \in (0, \infty)$, $x, y\in M$
and $\mathbf{a}, \mathbf{b} \in \mathcal{B} (M, \Delta)$, 
we have
\begin{align*}  
Z_M^{-2} \| I_{t}^{\mathbf{a}} (x)- I_s^{\mathbf{b}} (y) \|_{\ell^2}^2
&\le
p(t, x, x) +p(s, y, y) -2 p((t+s)/2, x, y) 
\\
&\qquad +2\rho (\mathbf{a}, \mathbf{b}) q^{(N)} (t,x,x)^{1/2}q^{(N)} (s,y,y)^{1/2},
\end{align*}
where we set 
\[
q^{(N)} (t,x,x):=\sum_{i=1}^\infty \lambda_i^{N/2} 
e^{-\lambda_i t} \vp_i^{\mathbf{a}} (x)^2.
\]
In particular, as a map from 
$(0, \infty) \times M\times \mathcal{B} (M, \Delta)$ to $\ell^2$, $I$
is continuous.
\end{proposition}

\begin{proof} 
First, note that the series that defines $q^{(N)}(t, x, x)$ converges 
absolutely and uniformly on $[t_0, \infty) \times M$ for every $t_0 >0$,
thanks to Lemma \ref{lem.int_EVCfunct} and Proposition \ref{prop.ef_est}.

By straightforward computation, we obtain 
\begin{align}  \label{eq.conti_love}
Z_M^{-2} \| I_{t}^{\mathbf{a}} (x)- I_s^{\mathbf{b}} (y) \|_{\ell^2}^2
&=
\sum_{i=1}^\infty 
  \{ e^{-\lambda_i t/2} \vp_i^{\mathbf{a}}  (x)-e^{-\lambda_i s/2} \vp_i^{\mathbf{b}}  (y)\}^2
\nn\\
&=
p(t, x, x) +p(s, y, y) -2 p((t+s)/2, x, y) -2A,
\nn
\end{align}
where we set
\[
A:=\sum_{i=1}^\infty  e^{-\lambda_i (t+s)/2}
 \{\vp_i^{\mathbf{a}} (x) \vp_i^{\mathbf{b}} (y) 
 -\vp_i^{\mathbf{a}} (x) \vp_i^{\mathbf{a}} (y) \}.
\]

We will calculate $A$. 
Consider each eigenspace $E_j$ and write $\nu_j :=\dim E_j$.
There exists unique $r_j \in \N$ such that
$\lambda_i =\lambda_j^\prime$ if and only if $r_j\le i \le r_j+\nu_j -1$.
Then, there exists 
$u_{ki} =u_{ki} (\mathbf{b}, \mathbf{a})$ for 
$i, k \in \{r_j , \ldots,  r_j+\nu_j -1\}$ which are 
constant in $x$ and  satisfies that
\[
\vp_i^{\mathbf{b}} = \sum_{k= r_j }^{r_j +\nu_j -1}  u_{ki} \vp_k^{\mathbf{a}},
\quad\qquad
r_j \le i \le r_j+\nu_j -1.
\]
Note that $\{ u_{ki} \mid  i, k \in \{r_j, \ldots,  r_j+\nu_j -1\}\}\in O(\nu_j)$.
From this we obtain 
\[
A=\sum_{j=1}^\infty  e^{-\lambda_j^\prime (t+s)/2}
  \sum_{i, k= r_j }^{r_j +\nu_j -1} 
\vp_i^{\mathbf{a}} (x) \vp_k^{\mathbf{a}} (y) 
\{  u_{ki} (\mathbf{b}, \mathbf{a}) -\delta_{ki}\},
\]
where $\delta_{ki}$ denotes Kronecker's delta.
Since $\tilde{\rho}_k$ is bi-invariant, we have 
\[
\sum_{i, k= r_j }^{r_j +\nu_j -1} | u_{ki} (\mathbf{b}, \mathbf{a}) - \delta_{ki}|^2
\le \rho_{E_j} ( \mathbf{a}\vert_{E_j}, \mathbf{b}\vert_{E_j})^2,
\qquad j\ge 1.
\]
It immediately follows from this and Schwarz' inequality that 
\begin{align*}
|A| 
&\le 
\sum_{j=1}^\infty  e^{-\lambda_j^\prime (t+s)/2} 
\left( \sum_{i= r_j }^{r_j +\nu_j -1}  \vp_i^{\mathbf{a}} (x)^2\right)^{\frac12}
\left( \sum_{k= r_j }^{r_j +\nu_j -1}  \vp_k^{\mathbf{a}} (y)^2\right)^{\frac12}
\rho_{E_j} ( \mathbf{a}\vert_{E_j}, \mathbf{b}\vert_{E_j})
\\
&\le
\rho (\mathbf{a}, \mathbf{b})
\sum_{j=1}^\infty (\lambda_j^\prime)^{N/2} e^{-\lambda_j^\prime (t+s)/2}  
\left( \sum_{i= r_j }^{r_j +\nu_j -1}  \vp_i^{\mathbf{a}} (x)^2\right)^{\frac12}
\left( \sum_{k= r_j }^{r_j +\nu_j -1}  \vp_k^{\mathbf{a}} (y)^2\right)^{\frac12}
\\
&\le
\rho (\mathbf{a}, \mathbf{b})
\left( 
\sum_{i=1}^\infty \lambda_i^{N/2} e^{-\lambda_i t}\vp_i^{\mathbf{a}} (x)^2
\right)^{\frac12}
\left( 
\sum_{i=1}^\infty \lambda_i^{N/2} e^{-\lambda_i s}\vp_i^{\mathbf{a}} (y)^2
\right)^{\frac12}
\\
&\le
\rho (\mathbf{a}, \mathbf{b}) q^{(N)} (t,x,x)^{1/2}q^{(N)} (s,y,y)^{1/2}. 
\end{align*}
Thus, we have obtained the desired estimate, from which the continuity 
of $I$ immediately follows.
\end{proof}

\section{Non-degeneracy of spectral distance}

Now we state our main theorem of this paper.
\begin{theorem} \label{thm.main}
Let $t>0$ and let $(M,\cD,g)$ and $(\hat{M}, \hat{\cD}, \hat{g})$
be two compact equiregular sub-Riemannian manifolds.
Then, $\mathbf{dist}_t (M, \hat{M})=0$ if and only if
these two manifolds are isometrically isomorphic 
as sub-Riemannian manifolds, that is, there exists 
a $C^\infty$-diffeomorphism $\psi =\psi_t\colon M \to \hat{M}$ such that 

\begin{itemize} 
\item
$\psi_* (\cD_x) = \hat{\cD}_{\psi (x)}$ for every $x\in M$.
\item
$\hat{g}_{\psi (x)} (\psi_* v, \psi_* v) = g_x (v, v)$
for  every $x\in M$ and $v\in \cD_x$.
\end{itemize}
In particular, $\mathbf{dist}_t$ is a distance on
the isometrically isomorphic 
class of compact equiregular sub-Riemannian manifolds
for each fixed $t >0$. 
\end{theorem}

\begin{proof} 
The ``if part" is trivial. We will prove ``only if part" below.

Suppose that 
$\sup_{\mathbf{b}\in \mathcal{B} (\hat{M}, \hat{\Delta})}
\inf_{\mathbf{a}\in \mathcal{B} (M, \Delta)} 
\mathrm{HD} (I_t^{\mathbf{a}} (M), I_t^{\mathbf{b}} (\hat{M}))  =0$.
This implies that for every $\mathbf{b}$, 
there exists a sequence $\{\mathbf{a}_k\}_{k\in \N}\subset \mathcal{B} (M, \Delta)$ such that
$\mathrm{HD} (I_t^{\mathbf{a}_k} (M), I_t^{\mathbf{b}} (\hat{M})) \to 0$
as $k\to\infty$.
Since $\mathcal{B} (M, \Delta)$ is compact, a subsequence,
which will be denoted by the same symbol again,
converges to a certain element $\mathbf{a}_\infty \in\mathcal{B} (M, \Delta)$.
 
From Proposition \ref{prop.diff_embed}, we can easily see that
\[
\sup_{x\in M}\inf_{y\in M}\| I_{t}^{\mathbf{a}} (x)- I_t^{\mathbf{b}} (y) \|_{\ell^2}
\le 
\sqrt{C_t \rho (\mathbf{a}, \mathbf{b})},
\]
where $C_t := 2Z_M^2  \sup_{x\in M}q^{(N)} (t,x,x) <\infty$ is a positive constant
independent of $(\mathbf{a}, \mathbf{b})$.
So, $\mathrm{HD} (I_t^{\mathbf{a}_k} (M), I_t^{\mathbf{a}_\infty} (M)) 
\le \sqrt{C_t \rho (\mathbf{a}_k, \mathbf{a}_\infty)} \to 0$ as $k\to\infty$.
By the triangle inequality for $\mathrm{HD}$, we have
$\mathrm{HD} (I_t^{\mathbf{a}_\infty} (M), I_t^{\mathbf{b}} (\hat{M})) = 0$.
Thus, we can find $(\mathbf{a}, \mathbf{b})$ such that 
$I_t^{\mathbf{a}} (M)=I_t^{\mathbf{b}} (\hat{M})$.
We fix such an $(\mathbf{a}, \mathbf{b})$ from now on and write 
$\mathbf{a}=\{\vp_i\}_{i=0}^\infty$ and $\mathbf{b}=\{\hat{\vp}_i\}_{i=0}^\infty$.
Thus, we have seen that 
\begin{itemize} 
\item
For every $x\in M$, there uniquely exists $\hat{y}_t \in \hat{M}$ such that
\begin{equation}\label{eq.0526_1}
\mu (M)^{1/2} e^{-\lambda_i t/2} \vp_i (x)
 =
\hat\mu (\hat{M})^{1/2} e^{-\hat\lambda_i t/2} \hat{\vp}_i (\hat{y}_t ), \qquad i\ge 1.
\end{equation}
\item
For every $\hat{y}\in \hat{M}$, there uniquely exists $x_t \in M$ such that
\begin{equation}\nn
\mu (M)^{1/2} e^{-\lambda_i t/2} \vp_i (x_t)
 =
\hat\mu (\hat{M})^{1/2} e^{-\hat\lambda_i t/2} \hat{\vp}_i (\hat{y}), \qquad i\ge 1.
\end{equation}
\end{itemize}
Define a homeomorphism $f_t\colon M \to \hat{M}$ by 
$x \mapsto \hat{y}_t$ and 
another homeomorphism $h_t\colon \hat{M}\to M$ by 
$\hat{y} \mapsto x_t$.  By definition, they are the inverses of each other.

Suppose that, at $x_0 \in M$, there exists a proper subspace $V$
of $T_{x_0}M$ which contains $\{ \nabla \vp_i (x_0) \mid i\ge 1\}$.
Here, $\nabla$ denotes the gradient operator on $M$ with respect to
any Riemannian metric that tames $g$.
Pick any $f \in C^\infty (M)$. 
By Lemma \ref{lem.int_EVCfunct}, Proposition \ref{prop.ef_est}
Lemma \ref{lem.Fourier} and Proposition \ref{prop.ef_est2},
the series \eqref{eq.Fourier} converges in $C^1$-topology, which implies that
\[
\nabla f (x_0) = \sum_{i=1}^\infty \langle f, \vp_i\rangle \nabla  \vp_i (x_0)\in V.
\]
This is a contradiction since $f$ is arbitrary.

Hence, for every fixed $x_0\in M$, we can find
$1\le i_1 < \cdots < i_d$ such that 
$\{ \nabla \vp_i (x_0) \mid i= i_1, \ldots, i_d\}$ spans $T_{x_0} M$
(recall that $d =\dim M$).
Define a smooth map $F_t\colon M\times \hat{M}\to \R^d$
by
\[
F_t (x, \hat{y}) =\{ \vp_{i_k} (x) -c_{i_k}(t) \hat\vp_{i_k} (\hat{y})  \}_{k=1}^d,
\quad
\mbox{where } c_{i_k}(t):=e^{ (\lambda_{i_k} -\hat\lambda_{i_k}) t/2}
\hat\mu (\hat{M})^{\frac12}\mu (M)^{-\frac12}.
\]
By way of construction, $ F_t ( h_t (\hat{y}), \hat{y})\equiv 0$.
Set $\hat{y}_0 = f_t (x_0)$, or equivalently $x_0 = h_t (\hat{y}_0)$.
If we apply $\nabla$ to the $x$-variable at $(x_0, \hat{y}_0)$, 
we obtain a linear isomorphism from $T_{x_0} M$ to $\R^d$.
From the implicit function theorem, the unique implicit function
on a neighborhood of $\hat{y}_0$,
which is necessarily smooth, must coincide with $h_t$. This proves
the smoothness of $h_t$ since $\hat{y}_0$ is actually arbitrary.
In the same way, the smoothness of $f_t$ can be shown.

The pushforward measure $(f_t)_* \mu$ can be written as 
$J_t d\hat\mu$, where $J_t$ stands for the Jacobian of $h_t$.
Integrating both sides of \eqref{eq.0526_1}, we have for all $i\ge 1$ that
\begin{align}  
0 &=
 \mu (M)^{1/2} e^{-\lambda_i t/2} \int_M\vp_i (x) \mu(dx)
 \nn\\
    &=
    \hat\mu (\hat{M})^{1/2} e^{-\hat\lambda_i t/2} 
     \int_M \hat{\vp}_i (f_t (x)) \mu(dx)
        \nn\\
    &=
    \hat\mu (\hat{M})^{1/2} e^{-\hat\lambda_i t/2} 
     \int_{\hat{M}} \hat{\vp}_i (\hat{y}) J_t (\hat{y})\hat\mu(dx).
     \nn
\end{align}
Since $J_t$ is orthogonal to $\hat{\vp}_i$ for all $i\ge 1$, 
$J_t$ is identical to the constant $\mu (M)/\hat\mu (\hat{M})$.
Integrating the square of \eqref{eq.0526_1} gives 
$\mu (M) e^{-\lambda_i t}=\hat\mu (\hat{M}) e^{-\hat\lambda_i t} J_t$,
which implies the sets of eigenvalues coincide, i.e. 
$\lambda_i = \hat\lambda_i $ for all $i\ge 0$.
Thus, we obtain 
\begin{equation}\label{eq.same_EF}
\vp_i (x)
 =
a  \hat{\vp}_i (\hat{y}_t )
\quad  \mbox{for all $i\ge 1$, \quad
where $a:= \hat\mu (\hat{M})^{1/2}  /\mu (M)^{1/2} >0$.}
\end{equation}

Next, we deduce the following intertwining relation
from \eqref{eq.same_EF}:
\begin{equation}\label{eq.intertwine}
\Delta \circ f_t^* = f_t^*\circ \hat{\Delta}.
\end{equation}
Pick any $k= \sum_{i=0}^\infty \langle k, \hat{\vp}_i\rangle \hat{\vp}_i  \in C^\infty (\hat{M})$. 
Then, 
$
 (\hat{\Delta}k) (\hat{y})=  \sum_{i=1}^\infty \langle k, \hat{\vp}_i\rangle \hat\lambda_i \hat{\vp}_i (\hat{y})
$
by Corollary \ref{cor.LapFour} and therefore
\begin{align}
(f_t^*\circ \hat{\Delta} k) (x) 
=
 (\hat{\Delta}k) (\hat{y}_t)
 =  
 \sum_{i=1}^\infty 
  \langle k, \hat{\vp}_i\rangle \hat\lambda_i \hat{\vp}_i (\hat{y}_t)
 =
\sum_{i=1}^\infty   \langle k, \hat{\vp}_i\rangle 
  \frac{\hat\lambda_i}{a} \vp_i (x), \quad x\in M.
  \nn
\end{align}
On the other hand, we can easily see that
\[
(f_t^* k) (x) =
 \sum_{i=0}^\infty \langle k, \hat{\vp}_i\rangle \hat{\vp}_i (\hat{y}_t)
=
\sum_{i=0}^\infty \langle k, \hat{\vp}_i\rangle \frac{1}{a}\vp_i (x).
\]
By a very similar argument to Corollary \ref{cor.LapFour}, we also have
\[
(\Delta \circ f_t^* k) (x) 
=\sum_{i=1}^\infty \langle k, \hat{\vp}_i\rangle 
\frac{\lambda_i}{a}\vp_i (x), \qquad x\in M.
\]
Thus, we have obtained \eqref{eq.intertwine}.


The intertwining relation \eqref{eq.intertwine} implies 
that $(f_t)_* g^* = \hat{g}^*$, that is, the two cometrics are preserved via the diffeomorphism $f_t\colon M \to \hat{M}$.
More precisely, for any 
$\psi, \chi \in C^\infty (\hat{M})$ vanishing at $\hat{y}_t=f_t (x)$,
we apply  the two operators in \eqref{eq.intertwine} to $-\psi \chi$
 and then evaluate it at $x$. Then, we obtain 
 \[
 g^*_{x} \la (J_f)^{\top} d\psi (\hat{y}_t), (J_f)^{\top} d\chi (\hat{y}_t)  \ra
 =
 \hat{g}^*_{\hat{y}_t} \la  d\psi (\hat{y}_t), d\chi (\hat{y}_t)\ra.
 \]
Here, $(J_f)^{\top}$ stands for the transpose of 
the Jacobian map $J_f\colon T_x M\to T_{\hat{y}_t} \hat{M}$ of $f$.
Since $\psi$ and $\chi$ are arbitrary, we have shown that 
\[
g^*_{x_t} \la (J_f)^{\top} \xi, (J_f)^{\top} \eta  \ra
 =
 \hat{g}^*_{\hat{y}} \la  \xi, \eta\ra, \qquad
 \xi, \eta \in T^*_{\hat{y}}\hat{M}, \,\, \hat{y}\in \hat{M}.
\]
Therefore, $M$ and $\hat{M}$ are isometrically isomorphic 
as sub-Riemannian manifolds. 
(In particular, $a=1$.)
This completes the proof of our main theorem (Theorem \ref{thm.main}).
\end{proof}

%
\appendix
\section{Appendix}

Let $(M,\cD,g)$ be a compact sub-Riemannian manifold of dimension $d$
with the rank of $\cD$ being $n~(1 \le n \le d)$.
(We {\it do not} assume the equiregularity in this appendix.)
Let $\mu$ be a smooth volume on $M$.
\footnote{
A measure on $M$ is said to be a smooth volume if its restriction to 
any coordinate chart is absolutely continuous with respect to 
the Lebesgue measure of the chart and the density is 
smooth and strictly positive.
}
The sub-Laplacian is defined by 
$\Delta =(\nabla^{\cD})^* \nabla^{\cD}$.
Here, $\nabla^{\cD}$ is the horizontal gradient in the direction of $\cD$
and $(\nabla^{\cD})^*$ is the adjoint of $\nabla^{\cD}$ with respect to $\mu$.
We study the heat semigroup $(e^{-t \Delta})_{t \ge 0}$
 associated with $ \Delta$.

The main aim of this appendix is to
 prove the following $L^\infty$-estimate for the first-order
derivatives of the heat semigroup. 
Our proof is probabilistic.
It should be noted that the exponent $\nu$ is independent of $A, f, t$.
\begin{proposition}\label{prop.Linfty_est}
Let the notation be as above. Then, there exists a constant $\nu >0$ 
with the following property: For every $A \in \Gamma (TM)$, there exists a
constant $C=C_A >0$ (which is independent of $f$ and $t$) such that 
\[
\| A e^{-t \Delta} f\|_{\infty} \le Ct^{-\nu} \| f\|_{\infty},
\qquad
f\in C^\infty (M), t \in (0,1].
\]
\end{proposition}

\begin{remark}\label{rem.onPDE}
PDE-theoretic analysis of the heat kernels associated with 
sum-of-squares operators under the H\"ormander condition 
was extensively studied in 1980's.
There are many nice results on the short time asymptotics/estimates of (derivatives of) the heat kernels.
Two prominent examples are \cite{jsc, sc}.
Using these results (after minor modification if necessary) might provide
a clearer proof for those who are familiar with that theory. 
The reason why Malliavin calculus is used in this appendix 
 is simply because the author is a probabilist.
\end{remark}

%

\subsection{Malliavin calculus}

The aim of this subsection is to recall the basics of Malliavin calculus
on the classical Wiener space.
The reader unfamiliar with this topic is referred to \cite{iwbk, mt, nu, sh}.
We only use standard results in this paper except that
our Wiener functionals take values in a manifold.
For manifold-valued Malliavin calculus, see Taniguchi \cite{ta}.

Let $({\cal W}, {\cal H}, \mathbb{P})$ be the $d$-dimensional
classical Wiener space, namely, 
\begin{itemize} 
\item
${\cal W}= \{ w \colon [0,1]\to \R^n \mid \mbox{$w$ is continuous 
and $w_0 =0$}\}$ is the Banach space of continuous functions 
from $[0,1]$ to $\R^d$ starting at $0$, which is equipped with the 
usual uniform norm.
\item
$\mathbb{P}$ is the $d$-dimensional Wiener measure on ${\cal W}$.
\item
${\cal H}$ is the Cameron-Martin space:
\begin{align*}
\cH
&=\{ h \in W
\mid \mbox{absolutely continuous and } 
\| h\|_{\cH}^2:=\int_0^1 | h^\prime_t|_{\R^n}^2 dt <\infty
\}.
\end{align*}
\end{itemize}
As is well-known, $\cH$ is a real separable Hilbert space 
and the coordinate process $(w_t)_{t \in [0,1]}$ is the standard
$d$-dimensional Brownian motion.

Now, we recall some definitions and basic facts
concerning Malliavin calculus on the 
classical Wiener space $({\cal W}, {\cal H}, \mathbb{P})$.
We often identify $\cH =\cH^*$ by the Riesz isometry as usual.

\begin{enumerate}
\item[{\bf (1)}]
We denote by ${\bf D}_{p,r} ({\cal K})$ the Sobolev spaces 
of ${\cal K}$-valued Wiener functionals
for the integrability index $p \in (1, \infty)$
and the differentiability index $r \in [0,\infty)$, where ${\cal K}$ is a real separable Hilbert space.
We set
${\bf D}_{\infty} ({\cal K})= \bigcap_{k=1 }^{\infty} \bigcap_{1<p<\infty} {\bf D}_{p,k} ({\cal K})$, which is the space of 
smooth Wiener functional.
When ${\cal K} ={\mathbb R}$, we write ${\bf D}_{p, r}$ and ${\bf D}_{\infty}$ for simplicity.
We denote by $D$ and $D^*$ the $\cH$-derivative
(i.e. the gradient operator in the sense of Malliavin calculus) 
and its adjoint i.e. (the minus of) the divergence operator, respectively.
$D$ is a bounded linear map from ${\bf D}_{p,r+1} ({\cal K})$ to 
${\bf D}_{p,r} ({\cal H}^* \otimes {\cal K})$ and 
$D^*$ is a bounded linear map from ${\bf D}_{p,r+1} ({\cal H}^* \otimes{\cal K})$ to ${\bf D}_{p,r} ({\cal K})$ for all $p \in (1,\infty)$
and $r \in [0, \infty)$.

\item
[{\bf (2)}]
For 
$F =(F^1, \ldots, F^d) \in {\bf D}_{\infty} ({\mathbb R}^d)$, we denote by 
$\sigma^{ij}_F (w) =  \la DF^i (w),DF^j (w)\ra_{{\cal H}^*}$
 the $(i,j)$-component of Malliavin covariance matrix ($1 \le i,j \le d$).
We denote by $\gamma^{ij}_F (w)$ the $(i,j)$-component of the inverse matrix $\sigma^{-1}_F$ (if it exists).
Recall that $F$ is called  non-degenerate in the sense of Malliavin
if $(\det \sigma_F)^{-1} \in  \bigcap_{1<p< \infty} L^p$.
Note that $\sigma^{ij}_F \in {\bf D}_{\infty}$ and
$D \gamma^{ij}_F = -\sum_{k,l} \gamma^{ik}_F ( D\sigma^{kl}_F ) \gamma^{lj}_F $.
Hence, derivatives of $\gamma^{ij}_F$ can be written in terms of
$\gamma^{ij}_F$'s and the derivatives of $\sigma^{ij}_F$'s,
which implies $\gamma^{ij}_F \in {\bf D}_{\infty}$, too.

\item
[{\bf (3)}] Let us recall the integration by parts 
 formula in the sense of Malliavin calculus for a non-degenerate 
 $F \in {\bf D}_{\infty} ({\mathbb R}^d)$
 (see \cite[p. 377]{iwbk}).
This formula plays a key role in this appendix.
Let $\psi \colon \R^d\to \R$ be a bounded $C^1$-function 
with bounded first-order partial derivatives.
Then, the following integration by parts formula
holds for every $G \in {\bf D}_{\infty}$:
\begin{align}
{\mathbb E} \bigl[
\partial_i \psi (F )  \cdot G 
\bigr]
=
{\mathbb E} \bigl[
\psi (F ) \cdot \Phi_i (\, \cdot\, ;G)
\bigr],
\label{ipb1.eq}
\end{align}
where $\partial_i$ stands for the $i$th partial differentiation
on $\R^d$ and $\Phi_i (w ;G) \in  {\bf D}_{\infty}$ is defined by 
\begin{align}
\Phi_i (w ;G) &= \sum_{j=1}^d D^* 
\left( \gamma^{ij }_F \cdot G \cdot DF^j  \right) (w).
\label{ipb2.eq}
\end{align}
\end{enumerate}

The proof of \eqref{ipb1.eq} is rather easy. Indeed,  we almost surely have
\begin{align}
\sum_{j=1}^d \langle D(\psi (F )), DF^j 
\rangle_{\cH^*} \gamma^{ij }_F 
&=
\sum_{j=1}^d \sum_{k=1}^d
\left\langle  \partial_k\psi (F ) DF^k, DF^j 
\right\rangle_{\cH^*} \gamma^{ij }_F  
\nn\\
&=
\sum_{k=1}^d \partial_k\psi (F ) 
\left( \sum_{j=1}^d \sigma^{kj}_F\gamma^{ij }_F \right) 
\nn\\
&=
\sum_{k=1}^d \partial_k\psi (F )  \delta_{ik} 
\nn\\
&=
\partial_i\psi (F ),
\label{ipb3.eq}
\end{align}
where $\delta_{ik}$ is Kronecker's delta.
By multiplying $G$, 
taking expectation and using the definition of $D^*$, we obtain \eqref{ipb1.eq}.

Let us quickly review manifold-valued Malliavin calculus.
Malliavin calculus for SDEs on manifolds was founded by 
Taniguchi \cite{ta}.
Roughly speaking, under suitable assumptions,
almost all of important results 
in the Euclidean case still hold true in the manifold case
with natural modifications.
\footnote{
In a recent book \cite{ta2} by Taniguchi himself, this topic 
(including the content of \cite{ta}) is 
explained in details. Unfortunately, this book is written in Japanese, however.
}

Let  ${\cal N}$ be a compact smooth manifold of dimension $d$.
Choose a Riemannian metric $g$ on $\cN$
so that the determinant of the Malliavin covariance 
of $\cN$-valued functionals are well-defined.
An $\cN$-valued Wiener functional 
$F\colon \cW \to \cN$ is said to belong to ${\bf D}_{\infty} ({\cal N})$
if $f (F) \in {\bf D}_{\infty}$ for every $f \in C^{\infty} (\cN)$.
If $\iota\colon \cN\hookrightarrow \R^k$ is an embedding, then 
$F \in {\bf D}_{\infty} ({\cal N})$ if and only if 
$\iota(F) \in {\bf D}_{\infty} (\R^k)$ since every $f \in C^{\infty} (\cN)$
extends to a smooth function on $\R^k$ with compact support.

For $F \in  {\bf D}_{\infty} (\cN)$, $D_h F (w) \in T_{F(w)} \cN$.
Hence, $D F (w) \colon \cH \to T_{F(w)} \cN$ is a bounded linear map
and $\la \alpha,  D F (w)\ra \in \cH^* =\cH$ for every 
$\alpha\in T^*_{F(w)} \cN$.
From the Riemannian metric $g_{F(w)}$ on $T_{F(w)} \cN$, 
we have $g^*_{F(w)}$ on $T^*_{F(w)} \cN$.
Take any ONB $\{e_i\}_{i=1}^d$ of $T^*_{F(w)} \cN$.
We set $DF^i (w) := \la e_i,  D F (w)\ra$ and 
\[
\det \sigma_F (w):=\det \left\{
 \la DF^i (w),DF^j (w)\ra_{{\cal H}^*}
 \right\}_{i,j=1}^d.
\]
Note that this definition is independent of the choice of the ONB,
but it does depend on the choice of $g$.
However, even if we choose another Riemmanian metric $\tilde{g}$,
there exists a constant $c>1$ such that 
\begin{equation}\label{eq.det_ctimes}
c^{-1} \det \tilde\sigma_F (w)
\le 
\det \sigma_F (w) \le c \det \tilde\sigma_F (w).
\end{equation}
For this reason, any choice of $g$ will do for our purpose.
We say that $F$ is called  non-degenerate in the sense of Malliavin
if $(\det \sigma_F)^{-1} \in  \bigcap_{1<p< \infty} L^p$.

%

We provide an integration by parts formula for a non-degenerate
manifold-valued Wiener functionals $F \in {\bf D}_{\infty} ({\cal N})$.
Below, we provide a ``hand-made version" of this formula
although a global version,
which looks geometrically beautiful, is also known.

Let $F \in {\bf D}_{\infty} ({\cal N})$ be non-degenerate in 
the sense of Malliavin and $\chi \in C^\infty (\cN)$.
Suppose that ${\rm supp}(\chi )$ is contained in a certain coordinate
neighborhood $U$.
The coordinate of $U$ is denoted by $(x^1, \ldots, x^d)$.
Let $\hat{\chi} \in  C^\infty (\cN)$ be such that 
$\hat{\chi} \equiv 1$ on ${\rm supp}(\chi )$ 
and ${\rm supp}(\tilde\chi )\subset U$.
If we set $h^i (x) :=\hat{\chi} (x) x^i$ and $F^i := h^i (F)$, 
then $h^i$ naturally extends to a smooth function on $\cN$ 
and $F^i \in {\bf D}_{\infty}$ ($1\le i \le d$).
Moreover, on $\{ w \in \cW\mid F(w) \in {\rm supp}(\chi )\}$, we have 
$F = (F^1, \cdots, F^d)$.

Similarly, if we set $\hat{\psi} := \psi \hat{\chi}$ for $\psi \in C^1 (\cN)$, 
then $\hat{\psi} \equiv \psi$ on ${\rm supp}(\chi )$.
Since ${\rm supp}(\hat\psi) \subset U$, it extends to a $C^1$-function 
on $\R^d$ with compact support, which will be denoted by the same symbol.
Then, we have 
\[
\psi (F) = \hat\psi (F^1, \ldots, F^d ) \qquad \mbox{on $\{F \in {\rm supp}(\chi )\}$.}
\]
Observe that, on RHS, an $\R^d$-valued Wiener functional is substituted into a $C^1$-function on $\R^d$.

Now, we denote by 
$\sigma^{ij}_F (w) =  \la DF^i (w),DF^j (w)\ra_{{\cal H}^*}$
 the $(i,j)$-component of Malliavin covariance matrix ($1 \le i,j \le d$).
Though $\sigma_F$ can be degenerate outside $\{F \in {\rm supp}(\chi )\}$,
it is non-degenerate on $\{F \in {\rm supp}(\chi )\}$ and satisfies 
that 
\[
(\det \sigma_F)^{-1} \cdot \mathbf{1}_{\{F \in {\rm supp}(\chi )\}} \in  \bigcap_{1<p< \infty} L^p. 
\] 
We denote by $\gamma^{ij}_F (w)$ the $(i,j)$-component of the inverse matrix $\sigma^{-1}_F$ on $\{F \in {\rm supp}(\chi )\}$.

\begin{remark}
In the above definition of $\sigma_F$, we used the standard metric
on $T^*\R^d$ through the obvious embedding $U \hookrightarrow \R^d$,
which is different from the original one $g^*$ on $T^* \cN$ used in the definition 
of non-degeneracy of $F$.
However, this does not matter
 since we still have a relation as in \eqref{eq.det_ctimes} on $\{F \in {\rm supp}(\chi )\}$.
\end{remark}

\begin{proposition}\label{prop.IbP_mfd}
Let the notation and situation be as above.  Then, we have
\begin{align}
{\mathbb E} \bigl[
\partial_i \psi (F )  \cdot \chi (F) G 
\bigr]
=
{\mathbb E} \bigl[
\psi (F ) \cdot \Phi_{i, \chi} (\, \cdot\, ;G)
\bigr],
\qquad 
G \in {\bf D}_{\infty}, \,\, 1 \le i \le d.
\label{ipb5.eq}
\end{align}
Here, {\rm (i)} $\partial_i$ is short hand for $\hat{\chi} (\partial/ \partial x^i)$
(which can be viewed as a vector field on $\cN$ or on $\R^n$)
and {\rm (ii)}  $\Phi_{i, \chi} (w ;G) \in  {\bf D}_{\infty}$ is defined by 
\begin{align}
\Phi_{i, \chi} (w ;G) &= \sum_{j=1}^d D^* 
\left( \gamma^{ij }_F \cdot \chi (F) G \cdot DF^j  \right) (w).
\label{ipb6.eq}
\end{align}
A precise definition of RHS on \eqref{ipb6.eq}
above will be given in \eqref{ipb9.eq}
below.
 \end{proposition}

\begin{proof}
Concerning \eqref{ipb5.eq}, one should 
note that the left hand side (LHS) does not depend on
the choice of  $\hat{\chi}$
and $\partial_i \psi (F )$ can be replaced by
 $\partial_i\hat\psi (F^1, \ldots, F^d)$, where $\partial_i$ stands for the standard partial differentiation on $\R^d$.

We will do a similar calculation to \eqref{ipb3.eq}. In this case,
however, the set $\{ \det \sigma_F =0\}$ may be of positive measure
and cause trouble.
For a clean proof, we introduce the following approximation.
For $m\in\N$, set $\sigma_F^m := \sigma_F +m^{-1} {\rm Id}_n$, where 
${\rm Id}_n$ is the identity matrix of size $n$.
Clearly, $\det \sigma_F^m \ge m^{-d}$ and therefore
$\gamma_F^m :=(\sigma_F^m)^{-1}$ exists for almost all $w$.
If $\det \sigma_F (w)>0 $, then $\gamma_F^m (w) \to \gamma_F (w)$ as 
$m\to\infty$.
Since $(\det \sigma_F^m )^{-1} \le (\det \sigma_F )^{-1}$, both
\[
\{ (\det \sigma_F^m )^{-1}\mathbf{1}_{\{F \in {\rm supp}(\chi )\}} \}_{m\in \N}
\quad\mbox{and}\quad
\{ |\gamma_F^m| \mathbf{1}_{\{F \in {\rm supp}(\chi )\}} \}_{m\in \N}
\]
 are bounded in $L^p$ for all $p\in (1,\infty)$.
(For the latter, think of adugate matrices.)

In a very similar way to \eqref{ipb3.eq}, we can compute
\begin{align}
\lefteqn{
\sum_{j=1}^d  {\mathbb E} \left[
\hat\psi (F^1, \ldots, F^d ) \cdot
D^* \left( \gamma^{m, ij }_F \cdot \chi (F) G \cdot DF^j  \right)
\right]
}
\nn\\
&=
\sum_{j=1}^d {\mathbb E} \left[
\langle D(\hat\psi (F^1, \ldots, F^d )), \, DF^j 
\rangle_{\cH^*} \gamma^{m, ij }_F \chi(F)G
\right]
\nn\\
&=
\sum_{j=1}^d \sum_{k=1}^d
{\mathbb E} \left[
  \partial_k\hat\psi (F^1, \ldots, F^d ) 
 \sigma^{kj }_F \gamma^{m, ij }_F \chi(F)G
 \right].
\label{ipb7.eq}
\end{align}
On $\{F \in {\rm supp}(\chi )\}$, $\lim_{m\to\infty}\gamma^{m, ij }_F =\gamma^{ij }_F$, a.s. and the integrand of RHS of 
\eqref{ipb7.eq} is uniformly integrable, which follows from 
$L^p$-boundedness for all $1<p<\infty$.
So, RHS converges to 
$
{\mathbb E} [\partial_i\hat\psi (F^1, \ldots, F^n ) \chi(F)G]
=
{\mathbb E} [\partial_i\psi (F) \chi(F)G]$, which equals LHS of 
\eqref{ipb5.eq}.

Next, we compute LHS of \eqref{ipb7.eq}. 
By the basic property of $D^*$, we can easily see that
\begin{align}
D^* \left( \gamma^{m, ij }_F \cdot \chi (F) G \cdot DF^j  \right)
&=
\gamma^{m, ij }_F \chi (F) G ( D^*DF^j )
-\gamma^{m, ij }_F \chi (F) \la DG,  DF^j  \ra_{\cH^*}
\nn\\
&
-\gamma^{m, ij }_F G\la D (\chi (F)),  DF^j  \ra_{\cH^*}
-\chi (F)G\la D \gamma^{m, ij }_F,  DF^j  \ra_{\cH^*}.
\label{ipb8.eq}
\end{align}
First, since $D$ is a kind of stochastic Gateaux derivative,
$D (\chi (F))$ vanishes outside $\{F \in {\rm supp}(\chi )\}$ and
so does RHS of \eqref{ipb8.eq}.
Hence, $\hat\psi (F^1, \ldots, F^n )$ on LHS of \eqref{ipb7.eq}
can be replaced by $\psi (F)$ as one can easily expect.
By the same argument as above the first three terms 
on RHS of \eqref{ipb8.eq} converges in $L^1$ to 
\[
\{ \gamma^{ij }_F \chi (F) G ( D^*DF^j )
-\gamma^{ij }_F \chi (F) \la DG,  DF^j  \ra_{\cH^*}
-\gamma^{ij }_F G\la D (\chi (F)),  DF^j  \ra_{\cH^*}\}
   \mathbf{1}_{\{F \in {\rm supp}(\chi )\}}
\]
as $m\to\infty$.  To deal with the fourth term, one should recall that
\[
D \gamma^{ij}_F = 
-\sum_{k,l=1}^d \gamma^{m, ik}_F ( D\sigma^{kl}_F ) \gamma^{m,lj}_F.
\]
By the same reasoning as above, the fourth term converges in $L^1$ to
\[
\chi (F)G \sum_{k,l=1}^d \la \gamma^{ik}_F ( D\sigma^{kl}_F ) \gamma^{lj}_F,  DF^j  \ra_{\cH^*}
  \mathbf{1}_{\{F \in {\rm supp}(\chi )\}}
\]
as $m\to\infty$. 
Thus, if we precisely define $\Phi_{i, \chi} (w ;G)$ in \eqref{ipb6.eq} by
\begin{align}
\Phi_{i, \chi} (\,\cdot\, ;G) &:=  \mathbf{1}_{\{F \in {\rm supp}(\chi )\}}\sum_{j=1}^d
\{
\gamma^{ij }_F \chi (F) G ( D^*DF^j )
-\gamma^{ij }_F \chi (F) \la DG,  DF^j  \ra_{\cH^*}
\nn\\
&\quad
-\gamma^{ij }_F G\la D (\chi (F)),  DF^j  \ra_{\cH^*}
+\chi (F)G \sum_{k,l=1}^d \la \gamma^{ik}_F ( D\sigma^{kl}_F ) \gamma^{lj}_F,  DF^j  \ra_{\cH^*}\},
\label{ipb9.eq}
\end{align}
we are done. 
(The right hand side of \eqref{ipb9.eq} should be understood to be $0$
outside $\{F \in {\rm supp}(\chi )\}$.)
\end{proof}


\begin{corollary}\label{cor.keyIbP}
Let the notation and situation be as in Proposition \ref{prop.IbP_mfd}. 
 Then, we have
\begin{align}
{\mathbb E} \bigl[
A \psi (F )  \cdot \chi (F) G 
\bigr]
=
{\mathbb E} \bigl[
\psi (F ) \cdot \Phi_{A, \chi} (\, \cdot\, ;G)
\bigr],
\qquad 
G \in {\bf D}_{\infty}, \, A\in \Gamma (\cN).
\label{ipb10.eq}
\end{align}
Here, $\Phi_{A, \chi} (\, \cdot\, ;G)$ is defined as follows: 
First, write $A (x) =\sum_{i=1}^d a^i (x) \partial_i$ on $U$ and 
set $\hat{a}^i = \hat{\chi} a^i$ so that $\hat{a}^i \in C^\infty (\cN)$ for all $1 \le i \le d$. Then, we set 
\begin{align}
\Phi_{A, \chi} (\, \cdot\, ;G) := \sum_{i=1}^d  \Phi_{i, \chi} (\, \cdot\, ;\hat{a}_i (F) G),
\label{ipb11.eq}
\end{align}
where $\Phi_{i, \chi}$ was defined in \eqref{ipb6.eq} and \eqref{ipb9.eq}.
\end{corollary}

\begin{proof}
Since LHS equals $\sum_{i=1}^d {\mathbb E} \bigl[\hat{a}_i (F)
\partial_i \psi (F )  \cdot \chi (F) G 
\bigr]$,
this corollary follows immediately from Proposition \ref{prop.IbP_mfd}. 
\end{proof}

%
\subsection{Frame bundle over sub-Riemannian manifold
and hypoelliptic diffusion process}

Since $-\Delta$ cannot be written as a sum-of-squares operator, 
we cannot construct the associated diffusion process 
as a solution of stochastic differential equation (SDE) on $M$.
However, if we consider a suitable SDE on a principal bundle over $M$, 
the projection of a solution is the desired  diffusion process on $M$.
This method is called Eells-Elworthy's construction. 
The classical case in the Riemannian setting is in \cite{hsu}.
A nice exposition on its sub-Riemannian version can be found in \cite{thal}.
Similar computations are in \cite{inatan, inatan2}, too.
Our exposition below is borrowed from \cite{ina_LDP}.

Take a Riemannian metirc $\hat{g}$ on $M$ which tames $g$,
that is, $\hat{g}|_{\cD\times\cD}=g$.
Any choice of $\hat{g}$ will do.
Denote by $\cD_x^{\perp}$ the orthogonal complement of $\cD_x$
in $T_xM$ with respect to $\hat{g}_x$ for $x\in M$ and
we set $\cD^{\perp} =\sqcup_{x\in M}\cD_x^{\perp}$, which is 
a subbundle of $TM$.

Now we introduce a principal bundle over $M$
with structure group $O (n) \times O(d-n)$, which is the product 
of two orthogonal groups, acting on it from the right.
\begin{align*}
\pfbM_x &= \{u\colon \R^n \oplus  \R^{d-n} =\R^d\to T_xM
 = \cD_x \oplus  \cD^{\perp}_x \mid
\\
& \qquad
\mbox{$u\restriction_{\R^n} \colon \R^n \to \cD_x$ 
and  $u\restriction_{\R^{d-n}} \colon \R^{d-n} \to \cD^{\perp}_x$ are linear isometries}
\},
\\
\pfbM &= \bigsqcup_{x\in \cM} \pfbM_x.
\end{align*}
This is a subbundle of the orthonormal frame bundle 
over the Riemannian manifold $(M, \hat{g})$.
For notational simplicity we will write $\cP := \pfbM$ and 
$G:= O (n) \times O(d-n)$.
The Lie algebra of $G$ is ${\frak o} (n) \times {\frak o} (d-n)$,
which will be denoted by ${\frak g}$.
Here, ${\frak o} (n)$ stands for the set of real $n \times n$ 
skew-symmetric matrices. 
The natural projection will be denoted by $\pi\colon \cP \to M$.

If we take a suitable Ehresmann connection $\omega$ on $\cP$, 
which is a ${\frak g}$-valued one-form on $\cP$,
the properties described below, including Relation \eqref{eq.8964-1},
are known to hold.
(Such an $\omega$ is concretely constructed in \cite[Section 3]{ina_LDP}.)
As usual, we define the horizontal subspace
$\cK_u\subset T_u \cP$ by
\[
    \cK_u=\{A\in T_u\cP\,|\,
          \omega_u(A_u)=0\}, \qquad u\in \cP.
\]
Then, the horizontal lift $\ell_u\colon T_{\pi(u)}M\to \cK_u$, which is a linear bijection,
is also defined uniquely and naturally.

Let $\{{\mathbf e}_i\mid 1\le i\le n\}$ be the canonical ONB of $\R^n$.
Define the canonical horizontal vector
fields $L_i$ on $\cP$ by 
$(L_i)_u=\ell_u(u{\mathbf e}_i)$ for $1 \le i \le n$.
Then, there exists a unique $V_0 \in \Gamma (\cD)$ such that,
for all $f \in C^\infty (M)$,
\begin{equation}\label{eq.8964-1}
 \left(
 \sum_{i=1}^n L_i^2 +L_0 \right) (f \circ \pi )
=
- (\Delta f) \circ \pi \qquad \mbox{on $\cP$,}
 \end{equation}
where $L_0 := \ell (V_0)$ is the horizontal lift of $V_0$.
Therefore, the law of the $(\sum_{i=1}^n L_i^2 +L_0)$-diffusion
process on $\cP$ starting at $u$ projects down to 
the law of the $(-\Delta)$-diffusion process on $M$ starting at $\pi (u)$.
Hence, for two starting points $u$ and $u^\prime$,
the law of the projected processes coincide if $\pi (u)= \pi (u^\prime)$.
Thanks to these facts, Eells-Elworthy's construction is available 
for the $(-\Delta)$-diffusion process on $M$.

As is well-known, 
we can realize the diffusion process as a solution of 
a stochastic differential equation (SDE).
Let $(w_t)_{t \in [0,1]}$ be a standard $n$-dimensional Brownian motion.
We consider the following SDE on $\cP$:
Let $(r(t, u))_{t\in [0,1]}$ be the unique solution to the following Stratonovich
SDE on $\cP$:
\begin{equation}\label{def.sde_OD}
    dr(t, u)=\sqrt{2}\sum_{i=1}^n L_i(r(t, u))\circ dw^i_t
       +L_0(r(t,u))dt,
    \quad r(0, u)=u\in \cP.
\end{equation}
By the reason we have just stated, we have a Feynman-Kac type 
representation: For all $f\in C^\infty (M)$ and $x\in M$, we have
\begin{equation}\label{eq.rep_FK}
e^{-t \Delta} f (x)
=
\mathbb{E} \left[ (f\circ \pi) ( r(t,u)) \right] \qquad \mbox{for every $u\in \pi^{-1}(x)$.}
\end{equation}
Since RHS (as a function of $u$) depends only on $\pi (u)$,
we see that
\begin{equation}\label{eq.rep_FK2}
Ae^{-t \Delta} f (x)
=
(\ell A)_u\mathbb{E} \left[ (f\circ \pi) ( r(t,u)) \right] \qquad
\mbox{for every $u\in \pi^{-1}(x)$.}
\end{equation}
Here, $\ell A$ is the horizontal lift of $A \in \Gamma (TM)$.
Hence, the proof of Proposition \ref{prop.Linfty_est} reduces to 
estimating the supremum of 
$(\ell A)\,\mathbb{E} \left[ (f\circ \pi) ( r(t,\,\cdot\,)) \right] $ over $\cP$.

By the H\"ormander condition on $\cD$, 
$\{L_1, \ldots, L_n\}$ satisfies the partial H\"ormander 
condition at every $u\in \cP$.
A precise statement of this condition is as follows: Set 
\[
\Sigma_1 :=\{L_1, \ldots, L_n\} \quad\mbox{and} \quad
\Sigma_k := \{ [V, L_i] \mid V\in \Sigma_{k-1}, \, 1\le i \le n\} 
\quad\mbox{for $k \ge 2$.} 
\]
Then, for every $u \in \cP$, the linear span of the subset 
$
\{ (\pi_*)_u W_u \mid W\in \bigcup_{k=1}^\infty  \Sigma_k\}
$
equals $T_{\pi (u)}M$.
According to \cite{ta}, the $M$-valued Wiener function  
$\pi (r(t,u))$, $t\in (0, 1]$ and $u\in \cP$, 
belongs to ${\bf D}_{\infty} (M)$ and non-degenerate in the sense of Malliavin.

Furthermore, since $\cP$ is compact, the partial H\"ormander 
condition is automatically uniform, that is,  there exists $K\in \N$
independent of $u$ such that
\[
\inf_{u\in \cP} \inf \left\{ \sum_{k=1}^K \sum_{W\in \Sigma_k}
\hat{g}_u ( (\pi_*)_u W_u, \eta)^2
~\middle|~ \eta \in T_{\pi (u)} M \mbox{ with } \hat{g}_u (\eta, \eta)=1\right\}>0.
\]
Thanks to the above condition, 
 a Kusuoka-Stroock type estimate is known to hold as in the Euclidean case:
 There exist positive constants $c_p$ and $\nu_1$ such that
\begin{equation}\label{ineq.KS_est}
\| \{\det \sigma_{\pi (r (t, u))} \}^{-1}  \|_{L^p} \le \frac{c_p}{t^{\nu_1}},   \qquad
t\in (0,1], \, u \in \cP, \, 1<p<\infty.
\end{equation}
Here, $\nu_1$ is independent of $(p, t, u)$ and $c_p$ is independent of $(t, u)$.


\begin{remark}
For non-experts of Malliavin calculus, 
we briefly explain how to show \eqref{ineq.KS_est} above.
First, we introduce a small parameter $\ve \in (0, 1]$
and consider a scaled version of SDE \eqref{def.sde_OD}:
\begin{equation}\label{def.sde_OD_scale}
    dr^{\ve}(t, u)=\sqrt{2}\ve \sum_{i=1}^n L_i(r^{\ve}(t, u))\circ dw^i_t
       +\ve^2 L_0(r^{\ve}(t,u))dt,
    \quad r^{\ve}(0, u)=u\in \cP.
\end{equation}
Put simply, $L_i$ is replaced by $\ve L_i$ ($1 \le i \le n$)
and $L_0$ by $\ve^2 L_0$.

By the scaling property of Brownian motion, 
the laws of $r^{\ve}(1, u)$ and $r(\ve^2, u)$ coincide 
and so do the laws of their Mallaivin covariance matrices.
Hence, \eqref{ineq.KS_est} is equivalent to 
the following estimate:
\begin{equation}\nn
\| \{\det \sigma_{\pi (r^{\ve} (1, u))} \}^{-1}  \|_{L^p} \le \frac{c_p}{\ve^{2\nu_1}},   \qquad
t\in (0,1], \, u \in \cP, \, 1<p<\infty.
\end{equation}
In order to show 
this inequality, we just need to 
repeat Taniguchi's proof in \cite{ta} for manifold-valued SDE \eqref{def.sde_OD_scale} with keeping track of 
the two parameters $\ve$ and $u$ carefully.
\end{remark}


\begin{remark} 
Among some variants of the sub-Riemannian Eells-Elworthy construction
 (see e.g. \cite{thal, inatan, inatan2, ina_LDP}), 
 we used one in \cite{ina_LDP}. 
The main reason is to avoid using a partial connection.
A drawback of using a partial connection is that $A \in \Gamma (TM)$
does not in general admit a horizontal lift unless $A \in \Gamma (\cD)$
and, consequently, Formula \eqref{eq.rep_FK2} breaks down.
\end{remark}

%
\subsection{Proof of Proposition \ref{prop.Linfty_est}}

The aim of this subsection is to prove Proposition \ref{prop.Linfty_est} by
carrying out Malliavin calculus (in particular, the integration by parts formula)
for the $\cP$-valued solution $(r(t,u))_{t\in [0,1]}$.
We remark that similar (or more complicated) computations were 
already done in the series of works \cite{inatan, inatan2, ina_LDP}.

We use a Nash embedding $\iota\colon M \hookrightarrow \R^N$
with respect to $\hat{g}$ for sufficiently large $N\in\N$.
(We will often identify $M$ and $\iota (M)$ and just write $M \subset \R^N$.)
If $f\in C^\infty (M)$, then $f$ extends to $\tilde{f}\in C^\infty (\R^N)$
with compact support with the following property at every $x\in M$: 
If $\mathbf{n}_x\in T_x \R^N \cong \R^N$ is normal to $T_xM \subset \R^N$,
then $\mathbf{n}_x \tilde{f}=0$ at $x$.
At every $x\in M$, $T_x \R^N$ admits a natural orthogonal decomposition
into the tangent and normal subspaces.
For $V \in \Gamma (T \R^N)$,  we define a vector field 
$V^{{\rm tan}}$ on $M$  by imposing $(V^{{\rm tan}})_x$ to be the 
tangent component to $T_xM$ of $V_x$.
Moreover, $V^{{\rm tan}}$ is smooth, i.e. $V^{{\rm tan}} \in\Gamma (M)$.
Then, $V  \tilde{f} \equiv (V^{{\rm tan}}) f$ on $M$.

Next, we introduce a beautiful embedding of $\cP = \pfbM$.
Let $\{{\mathbf e}_i\mid 1\le i\le d\}$ be the canonical ONB of $\R^d$.
For $u\in \cP$, we set 
\[
\iota^\prime (u) :=\left( \iota (\pi (u)); \, (\iota_*)_{\pi (u)} (u{\mathbf e}_1 ), \ldots,  (\iota_*)_{\pi (u)} (u{\mathbf e}_d)
\right)
\,\, \in \R^N \times (\R^N)^d =\R^{N(1+d)}.
\]
Since $\iota$ is an isometric embedding, 
$\{(\iota_*)_{\pi (u)} (u{\mathbf e}_i) \}_{i=1}^d$ is orthonormal in $\R^N$.
Then, $\iota^\prime\colon \cP \hookrightarrow \R^{N(1+d)}$ 
is an embedding. Moreover, these embeddings respect the structure 
of projections, that is, $\iota \circ \pi = \pi^\prime \circ \iota^\prime$.
Here,  
$\pi^\prime \colon \R^N \times (\R^N)^d \to \R^N$ 
is the projection that picks up the leftmost component of $\R^N \times (\R^N)^d =\R^{N(1+d)}$.
Thanks to this fact, we need not distinguish $\pi$ and $\pi^\prime$ 
when we identify $\cP$ and $\iota^\prime(\cP)$,
which makes our computations quite simple.
When there is no fear of confusion, we will identify 
$\iota^\prime(\cP) = \cP$ and write $\pi^\prime =\pi$ for simplicity
of notation.
(For the orthonormal  bundle over a Riemannian manifold, 
this kind of embedding is well-known. Our version here is a slight modification.)
In coordinates,  a generic element 
$u \in \R^N \times (\R^N)^d$ is denoted by $u=(u_{kl})_{0\le k \le d, 1\le l\le N}$ and  $\pi^\prime u=(u_{0l})_{1\le l\le N}$.


Any smooth vector field on $\cP$ extends to one on $\R^{N(1+d)}$ 
with compact support, which will be denoted by the same symbol again
(any such extension will do for our purpose).
Then, SDE \eqref{def.sde_OD} can be realized in $\R^{N(1+d)}$:
\begin{equation}\label{def.sde_OD2}
    dr(t, u)=\sqrt{2}\sum_{i=1}^n L_i(r(t, u))\circ dw^i_t
       +L_0(r(t,u))dt,
    \quad r(0, u)=u\in \R^{N(1+d)}.
\end{equation}
When $u \in \iota^\prime(\cP)$, the solution coincides with that of \eqref{def.sde_OD}.
For this reason, we will use the same symbol $(r(t, u))_{t\in [0,1]}$ again
by slightly abusing the notation.
Since $\ell A$ extends to a vector field on the ambient space
with compact support, there is a constant $C>0$ such that
\begin{equation}\label{eq.starbucks}
\sup_{u\in \cP} |(\ell A)_u\,\mathbb{E} \left[ (f\circ \pi ) ( r(t,u)) \right]  |
\le C 
\sup_{k,l}
\sup_{u\in \R^{N(1+d)}} \left| \frac{\partial}{\partial u_{kl}}
\mathbb{E} \left[ (\tilde{f} \circ \pi )( r(t,u)) \right] 
\right|.
\end{equation}
Hence, it suffices to estimate RHS of \eqref{eq.starbucks} above.


Since SDE \eqref{def.sde_OD2} is on a Euclidean space, it has been 
thoroughly studied. We recall some results now.
Denote by $(J (t,u))_{t \in [0,1]}$ the Jacobian process of $(r (t,u))_{t \in [0,1]}$,
namely, $J (t,u)= \nabla r (t,u)$. Here, $\nabla$ stands for 
the standard gradient on $\R^{N(1+d)}$ with respect to the $u$-variable.
Firstly, for all $p \in (1, \infty)$ and $r \in [0,\infty)$, we have
\begin{equation}\label{ineq.SDE1}
\sup_{u\in \R^{N(1+d)}} \sup_{t\in [0,1]} \{
\|r(t,u)\|_{{\bf D}_{p,r}} + \|J (t,u)\|_{{\bf D}_{p,r}} 
\}<\infty.
\end{equation}
Secondly, we also have
\begin{equation}\label{ineq.SDE2}
 \mathbb{E} \left[ \sup_{u\in \R^{N(1+d)}} |J (t,u) |
\right] <\infty, 
\qquad 0\le t \le 1.
\end{equation}
Note that the sup is inside the expectation in \eqref{ineq.SDE2}.

Take a partition of unity $1 \equiv \sum_{m =1}^K \chi_m$ 
such that ${\rm supp} (\chi_m)$ is contained in a coordinate chart
for each $m$.
By \eqref{ineq.SDE2} and Lebesgue's dominated convergence theorem,
one can see that
\begin{align}
\lefteqn{
\partial_{kl} \mathbb{E} \left[ (\tilde{f} \circ \pi )( r(t,u)) \right] 
}
\nn\\
&=
\mathbb{E} \left[\partial_{kl} \bigl\{ (\tilde{f} \circ \pi )( r(t,u)) \bigr\}\right]
\nn\\
&=
\sum_{j=1}^N 
\mathbb{E} \left[ (\partial_{0j} \tilde{f})  (\pi( r(t,u)) )\cdot  
J_{0j, kl} (t,u)\right]
\nn\\
&=
\sum_{j=1}^N \sum_{m=1}^K 
\mathbb{E} \left[ (\partial_{0j}^{{\rm tan}} f)  (\pi( r(t,u))) \cdot  
\chi_m  (\pi( r(t,u)))
\cdot
J_{0j, kl} (t,u)\right],
\label{ineq.SDE3}
\end{align}
where $\partial_{kl}  =\partial/\partial u_{kl}$ and 
$J_{0j, kl} (t,u)$ denotes the matrix component of $J (t,u)$.
To check the second equality above, just recall that 
$\pi( r(t,u)) = \{ r(t,u)_{0j}\}_{1 \le j\le N}$.

Since $\partial_{0j}^{{\rm tan}} \in \Gamma (M)$, 
$J_{0j, kl} (t,u) \in {\bf D}_\infty$ and $\pi( r(t,u)) \in {\bf D}_\infty (M)$ 
is non-degenerate, we can apply Corollary \ref{cor.keyIbP} 
(and Proposition \ref{prop.IbP_mfd}) to 
RHS of \eqref{ineq.SDE3} by setting
\[
F = \pi( r(t,u)),\,\,  G =J_{0j, kl} (t,u), \,\, A = \partial_{0j}^{{\rm tan}}, \,\,
\psi =f, \,\, \chi =\chi_m.
\]
We will apply H\"older's inequality to RHS on \eqref{ipb9.eq}
with these data to estimate \eqref{ineq.SDE3}.
In the explicit expression of the ``Malliavin weight" $\Phi_{A, \chi}$ in 
\eqref{ipb9.eq} and \eqref{ipb11.eq}, all but one factor 
are $L^p$-bounded in $(t, u)$ for every $p \in (1,\infty)$,
thanks to \eqref{ineq.SDE1}.
The only exception is $\gamma_F$, the inverse of Malliavin covariance matrix.
As for $\gamma_F$, it holds that
\[
|\gamma_F| \le c^\prime (\det \sigma_{F})^{-1} |\sigma_{F}|^{d-1}
\le 
 c^{\prime\prime}  (\det \sigma_{F})^{-1} \|D F\|_{\cH^*}^{2(d-1)},
\]
where the constant $c^\prime, c^{\prime\prime} >0$ does not depend on $F$.
\footnote{
The inverse of a matrix $Z$ equals $(\det Z)^{-1}$ times
the adjunct (adugate) matrix $\mathrm{adj} (Z)$ of $Z$.
}
Thanks to the Kusuoka-Stroock estimate \eqref{ineq.KS_est}, which is uniform in $(t,u)$, we have 
\[ 
\| \gamma_{\pi( r(t,u))} \|_{L^p} \le \frac{c_{1}}{t^{\nu_1}},
 \qquad
t\in (0,1], \,\, u \in \cP, \,\, 1<p<\infty.
\]
Here, $c_1$ and $\nu_1$ are certain positive constants 
independent of $(t, u)$.
Combining all arguments above, we can see that
RHS of \eqref{ineq.SDE3} is dominated by 
$c_2/t^{\nu_2}$ for all $t\in (0,1]$ and $u \in \cP$,
where $c_2$ and $\nu_2$ are certain positive constants 
independent of $(t, u)$.
This completes the proof of  Proposition \ref{prop.Linfty_est}.

\bigskip
\noindent
{\bf Acknowledgement:}~
The author is grateful to Professor Shouhei Honda for helpful comments.


\bigskip
\begin{flushleft}
  \begin{tabular}{ll}
    Yuzuru \textsc{Inahama}
    \\
    Faculty of Mathematics,
    \\
    Kyushu University,
    \\
    744 Motooka, Nishi-ku, Fukuoka, 819-0395, JAPAN.
    \\
    Email: {\tt inahama@math.kyushu-u.ac.jp}
  \end{tabular}
\end{flushleft}

\end{document}